\documentclass{amsart}
\usepackage{latexsym,amssymb,amsthm,amsmath}

\usepackage{amssymb}
\usepackage{amsmath}

\theoremstyle{plain}
\newtheorem{theorem}{Theorem}[section]
\newtheorem*{Theorem B}{Theorem B}
\newtheorem*{Theorem A}{Theorem A}
\newtheorem{lemma}{Lemma}[section]
\newtheorem{proposition}{Proposition}[section]
\newtheorem{corollary}{Corollary}[section]

\newtheorem{definition}{Definition}[section]
\newtheorem{example}{Example}[section]
\numberwithin{equation}{section}

\theoremstyle{remark}
\newtheorem{remark}{Remark}[section]

\def\({\left( }
\def\){\right)}
\def\<{\left< }
\def\>{\right> }

\title[Pointwise semi-slant warped products]{Geometry of pointwise semi-slant warped products in locally conformal Kaehler manifolds}

\author{Bang-Yen Chen}
\address{B.-Y. Chen: Department of Mathematics, Michigan State University, 619 Red Cedar Road,   East Lansing, Michigan 48824--1027, U.S.A.}
\email{chenb@msu.edu}
\author{Fatimah Alghamdi}\address{F. Alghamdi: Department of Mathematics, Faculty of Science, Jeddah University, 21589 Jeddah, Saudi Arabia}
\email{fmalghamdi@uj.edu.sa}
\author{Siraj Uddin}
\address{S. Uddin: Department of Mathematics, Faculty of Science, King Abdulaziz University, 21589 Jeddah, Saudi Arabia}
\email{siraj.ch@gmail.com}


\begin{document}
\begin{abstract} 
 In this paper, we study the geometry of pointwise semi-slant warped products in a locally conformal Kaehler manifold. In particular, we obtain several results which extend Chen's inequality for $CR$-warped product submanifolds in Kaehler manifolds. Also, we study the corresponding equality cases. Several related results on pointwise semi-slant warped products are also proved in this paper.  
    
\end{abstract}

\maketitle

\sloppy

\section{Introduction}

The notion of slant submanifolds of an almost Hermitian manifold were introduced by first author in \cite{C90,C90b} which include both totally real and holomorphic submanifolds. Since then many papers on these submanifolds have been published (see, e.g., \cite{CSA21.1,CSA21.2}). Also, as the generalizations of 
totally real and holomorphic submanifolds, A. Bejancu introduced the notion of CR-submanifolds in \cite{Be78}. CR-submanifolds have also been studied by many geometers (see, e.g. \cite{Be79,Be86,BC79,C81.1,C81.2}).

As a generalization of slant submanifolds,  N. Papaghiuc \cite{Papa94}  introduced the notion of semi-slant submanifolds of an almost Hermitian manifold which includes the classes of CR-submanifolds and slant submanifolds.
As another extension of slant submanifolds, F. Etayo \cite{Etayo} defined the notion of pointwise slant submanifolds of almost Hermitian manifolds under the name of quasi-slant submanifolds. Then, the first author and O. J. Garay  studied in \cite{C12} pointwise slant submanifolds of Kaehler manifolds and they proved several fundamental results of such submanifolds and provided a method to construct examples of such submanifolds. 

In the 1960s, R.L. Bishop and B. O'Neill defined warped product manifolds. It is well-known that warped product manifolds play important role in differential geometry as well as in physics. At the beginning of this century, the first author initiated  in \cite{C01a,C01b,C02} the study of warped products  in Riemannian and Kaehler manifolds from the submanifold point of view.  He proved in \cite{C01a} that there do not exist warped products of the form:
$N^{\perp}\times_f N^{T}$  in any Kaehler manifold beside $CR$-products,  where
$N^{\perp}$ is a totally real submanifold and $N^{T}$ is a holomorphic submanifold. He also shown in \cite{C01a} that there exist many $CR$-submanifolds which are  warped products of the form $N^{T}\times_f N^{\perp}$ {\it by reversing the two factors $N^{T}$ and $N_\perp$}. He simply called such
$CR$-submanifold a $CR$-{\it warped product}.  
Furthermore, he proved that every $CR$-warped product $N^{T}\times_f N^{\perp}$ in any Kaehler manifold satisfies the basic inequality: 
\begin{align}\label{1.1} ||h||^2\geq 2p\|\vec\nabla(\ln f)\|^2,\end{align} where $f$ is the warping function, $p$ is the dimension of $N^{\perp}$, $||h||^2$ the squared norm  of the second fundamental form, and
$\vec\nabla(\ln f)$  the gradient of $\ln f$. 
Since then the geometry of warped product submanifolds becomes an active  research subject (for more details, we refer to Chen's books \cite{book, book17} and his survey article \cite{C13}). 

In \cite{Sahin06}, B. Sahin proved that there do not exist semi-slant warped products in a Kaehler manifold. Later, he investigated warped product pointwise semi-slant submanifolds of Kaehler manifolds in \cite{Sahin13}. Further, V, Bonanzinga and K. Matsumoto \cite{BK04} studied $CR$-warped product submanifolds and semi-slant warped product submanifolds in locally conformal Kaehler manifolds either of the form $N^T\times_fN^\perp$ and or of $N^T\times{_{f}N^\theta}$, where $N^{T}, N^{\perp}$  and $N^{\theta}$ are  holomorphic, totally real and slant submanifolds, respectively (see also \cite{K17}). Since then, such warped product submanifolds have been studied by many geometers (see, e.g.,  \cite{Al1,book17}).

In this paper, we investigate pointwise semi-slant warped products in locally conformal Kaehler manifolds. In particular, we obtain several results which extend Chen's inequality \eqref{1.1}; and we also study the corresponding equality cases. Several related results  are also established in this paper.


\section{Preliminaries}\label{Sec2}

A {\it locally conformally Kaehler manifold} $(\tilde M,J,g)$ of (or an $LCK$-manifold for short) is a complex manifold $(\tilde{M}, J)$ endowed with a Hermitian metric $g$ which is locally conformal to a Kaehlerian metric. Equivalently,  there
exists an open cover  $\{U_{i}\}_{i\in I}$ of $\tilde M$ and a family $\{f_{i}\}_{i\in I}$ of real-valued differentiable functions $f_{i} : U_{i} \to \mathbb{R}$ such that $g_{i}= e^{-f_{i}} g|_{U_{i}}$ is a Kaehlerian metric on $U_{i}$, i.e., $\nabla^{*}J = 0$, where $J$ is the almost complex structure, $g$ is the Hermitian metric, and $\nabla^{*}$ is the covariant differentiation with respect to $g$. A typical example of a compact $LCK$-manifold is a Hopf manifold which is diffeomorphic
to $\mathbb{S}^{1} \times\mathbb{S}^{2n-1}$ and it admits no Kaehler structure (see \cite{Va76}). 
Let $\Omega$ and $\Omega_{i}$ denote the 2-forms associated with $(J,g)$ and  $(J,g_{i})$, respectively (i.e., $\Omega(X,Y)=g(X,JY)$, etc.). Then $\Omega_{i}=e^{-f_{i}}\Omega|_{U_{i}}$. 

The following result from \cite{Va76} is well-known  (see also \cite[Theorem 1.1]{DS98}).

\begin{theorem}\label{T1} The Hermitian manifold $(\tilde{M},J,g)$ is an $LCK$-manifold if and only if there exists a globally defined closed $1$-form $ \alpha$ on $\tilde{M}$ such that $d\Omega=\alpha\wedge \Omega$.
\end{theorem} 

 The closed 1-form $\alpha$ in Theorem \ref{T1} is called the {\it Lee form} and the vector field $\lambda=\alpha^{\#}$ dual to $\alpha$, (i.e., $g(X,\lambda)=\alpha(X)$ for $X\in T\tilde M$), is called  the {\it Lee vector field} of the $LCK$-manifold $\tilde M$. 
 An $LCK$-manifold $(\tilde {M}, J, g)$ is called a {\it globally conformal Kaehler manifold} ($GCK$-manifold for short)  if one can choose $U=\tilde M$. An $LCK$-manifold  $\tilde{M}$ is a $GCK$-manifold if and only if the 1-form $\alpha $  is exact.

If $\tilde{\nabla}$ denotes the Levi-Civita connection on an $LCK$-manifold  $\tilde{M}$, then we have
\begin{align}\label{2.1}
(\tilde{\nabla}_X J)Y = -g(\beta ^{\#},Y )X -g(\alpha^{\#}, Y )JX + g(JX, Y )\alpha ^{\#} +g(X, Y )\beta^{\#} .
\end{align}
for tangent vector fields $ X, Y$ on $\tilde M$, where $ \alpha^{\#}$ is  the Lee vector field, $\beta $ is the 1-form defined by
$\beta(X) =-\alpha(JX)$ for any $X$ $\in T \tilde M$ and $\beta^{\#}$ denotes the dual vector field of $\beta$ (see \cite{Va76}). In terms of the Lee vector field, equation \eqref{2.1} can be written as
\begin{align}\label{2.2}
(\tilde\tilde{\nabla}_X J)Y = [g(\lambda, JY )X-g(\lambda, Y )JX + g(JX, Y )\lambda + g(X, Y )J\lambda].
\end{align}

$LCK$-manifolds  contain rich source since their Lee-form plays important role in determining several geometric features of their submanifolds.
An $LCK$-manifold $(\tilde{M}, J, g)$ is called a {\it Vaisman manifold} if its Lee form is parallel, i.e., $\tilde\nabla \alpha = 0$ (see \cite{Va76,Va80}). Vaisman manifolds form the most important subclass of $LCK$-manifolds.

Let $(\tilde{M}, J, g,\alpha)$ be a complex $m$-dimensional $LCK$-manifold and $M$ be a real $n$-dimensional Riemannian manifold isometrically immersed in $\tilde{M}$ with $n\leq m$. Denote by $\Gamma(TM)$ and $\Gamma(T^\perp M)$ the spaces of tangent and normal vector fields of $M$, respectively. Let $g$ also denote the induced metric tensor on $M$ and by $\nabla$ the covariant differentiation with respect to the induced metric on $M$. Then the Gauss and
Weingarten formulas for $M$ are given respectively by
\begin{align}
\label{2.3}
&\tilde\nabla_XY=\nabla_XY+h(X, Y),
\\&\label{2.4}
\tilde {\nabla}_X\xi=-A_\xi X+ \nabla^{\perp}_X\xi,
\end{align}
for any $X, Y\in\Gamma(TM)$ and $\xi\in\Gamma(T^\perp M)$, where $\nabla^\perp$ is the normal connection,  $h$  is the second fundamental form,  and $A_\xi$ is the Weingarten map (or the shape operator) associated with $\xi$. It is well-known that $A_\xi$ and $h$ are related by
\begin{align}
\label{2.5}g(h(X, Y), \xi)=g(A_\xi X, Y).\end{align}
 A submanifold  $M$ is said to be {\it totally geodesic} if its second fundamental form $h$ vanishes identically, i.e., $h = 0$, or equivalently $A = 0$.

For any vector $X$ tangent to $M$, we write
\begin{align}
\label{2.6} JX=PX+F X,\end{align}
where $PX$ and $FX$ are the tangential and normal components of $JX$, respectively. Similarly, for any vector $\xi$ normal to $M$, we put
\begin{align}
\label{2.7}J\xi=t\xi+f\xi,\end{align}
where $t\xi$ and $f\xi$ are the tangential and normal components of $J\xi$, respectively.

 It was known in \cite{C12} that a submanifold $M$ of an almost Hermitian manifold $\tilde M$ is {\it pointwise slant} if and only if
 \begin{align}\label{2.8} P^2=-(\cos^2\theta)I,\end{align}
for some real-valued function $\theta$ defined on $M$, where $I$ is the identity map of the tangent bundle $TM$ of $M$. A pointwise slant submanifold is called {\it proper} if it does not contain any totally real or complex points, i.e., $0<\cos^2\theta<1$.

At a given point $p\in M$, the following relations are easy consequences of (\ref{2.8}):
\begin{align}
\label{2.9}
g(PX, PY)=(\cos^2\theta)\,g(X, Y),
\end{align}
\begin{align}
\label{2.10}
g(FX, FY)=(\sin^2\theta)\,g(X, Y)
\end{align}
for any $X, Y\in\Gamma(TM)$.
Further, it is easy to verify that
\begin{align}
\label{2.11}
tFX=-\sin^2\theta\,X,\quad fFX=-FPX
\end{align}
for any $X\in\Gamma(TM)$.

\begin{definition}\label{D2} \rm{Let $\tilde M$ be an almost Hermitian manifold and $M$ be a submanifold of $\tilde M$. Then $M$ is called a {\it pointwise semi-slant submanifold} if there exists a pair of orthogonal distributions ${\mathfrak{D}}$ and ${\mathfrak{D}}^\theta$ on $M$ such that
\begin{enumerate}
\item [(i)] The tangent bundle $TM$ is the orthogonal decomposition $TM={\mathfrak{D}}\oplus{\mathfrak{D}^\theta}$.
\item [(ii)]  The distribution ${\mathfrak{D}}$ is $J$-invariant (or holomorphic), i.e., $J({\mathfrak{D}})={\mathfrak{D}}.$
\item [(iii)]  The distribution ${\mathfrak{D}^\theta}$ is pointwise slant with slant function $\theta$.
\end{enumerate}}
The pointwise semi-slant submanifold $M$ is called {\it proper} if  neither $\dim\, N^T=0$ nor the slant function of ${\mathfrak{D}}^\theta$ is $\frac{\pi}{2}$, i.e., $\cos^2\theta>0$. Otherwise, $M$ is called {\it improper}. 
\end{definition}

\begin{definition}\label{D3} \rm{Let $M$ be a submanifold of an almost Hermitian manifold. Then
\begin{enumerate}
\item [(i)] $M$ is called {\it $\mathfrak{D}$-geodesic} if $h(X, Y) = 0$,  
\item [(ii)]  $M$ is called {\it ${\mathfrak{D}^\theta}$-geodesic} if $h(Z,W) = 0$, 
\item [(iii)] $M$ is called {\it mixed totally geodesic} if $h(X, Z) = 0$,
\end{enumerate}}
\noindent for all $X,Y\in\Gamma (\mathfrak{D})$ and for all $Z, W\in\Gamma (\mathfrak{D}^\theta)$.
\end{definition}

\section{Some lemmas}\label{Sec3}

First, we prove the following results for later use.

\begin{lemma}\label{L3.1}
 Let $ M$ be a proper pointwise semi-slant submanifold of an $LCK$-manifold. Then, for any $ X, Y \in\Gamma (\mathfrak{D})$ and $Z \in\Gamma (\mathfrak{D}^\theta)$, we have
\begin{align*}
\label{2.13}
\sin^2\theta g(\nabla_XY,Z)=\,&g(A_{FZ}JY-A_{FPZ}Y,X)-g(JX,Y)g(\alpha^{\#},FZ)\\&-g(X,Y)g(\beta^{\#},FZ).
\end{align*}
\end{lemma}
\begin{proof}
For any $X,Y\in\Gamma (\mathfrak{D})$ and $Z\in\Gamma (\mathfrak{D}^\theta)$, we have
\begin{align*}
g(\nabla_XY, Z)=g(J\tilde\nabla_XY, JZ)=g(J\tilde\nabla_XY, FZ)+g(J\tilde\nabla_XY, PZ).
\end{align*}
From the covariant derivative formula of $J$, we derive
\begin{align*}
g(\nabla_XY, Z)=g(\tilde\nabla_XJY, FZ)-g((\tilde\nabla_XJ)Y, FZ)-g(\tilde\nabla_XY, P^{2}Z)-g(\tilde\nabla_XY, FPZ).
\end{align*}
Then, using (\ref{2.1}), (\ref{2.5}) and (\ref{2.8}), we arrive at
\begin{align*}
g(\nabla_XY, Z)=\,&g(A_{FZ}JY, X)-g(JX,Y)g(\lambda, FZ)-g(X,Y)g(J\lambda,FZ)\notag\\
&-g(Y,\tilde\nabla_X\cos ^{2}\theta Z)-g(A_{FPZ}Y, X).
\end{align*}
Since $M$ is a proper pointwise semi-slant submanifold ,we get
\begin{align*}
g(\nabla_XY, Z)=\,&g(A_{FZ}JY-A_{FPZ}Y, X)-g(JX,Y)g(\lambda, FZ)-g(X,Y)g(J\lambda,FZ)
\\&-\cos ^{2}\theta g(\tilde\nabla_XZ,Y)+\sin2\theta X(\theta)g(Y,Z).
\end{align*}
Thus, the lemma follows from above relations by using the orthogonality of the two distributions. 
\end{proof}

If ${\mathfrak{D}}$ is a totally geodesic distribution in $M$, then $g(\nabla_XY,Z)=0$ for any $X,Y\in\Gamma (\mathfrak{D})$ and $Z\in\Gamma (\mathfrak{D}^\theta)$. Hence, Lemma \ref{L3.1} implies the following.

\begin{lemma}\label{L3.2}    
Let $M$ be a proper pointwise semi-slant submanifold of an $LCK$-manifold  $\tilde M$. Then  the distribution ${\mathfrak{D}}$ defines a totally geodesic foliation if and only if
\begin{align*}
g(A_{FZ}J X-A_{FPZ}X, Y)=g(X,Y)g(\beta^{\#},FZ)-g(JX,Y)g(\alpha^{\#},FZ),
\end{align*}
for any $X,Y\in\Gamma (\mathfrak{D})$ and $Z\in\Gamma (\mathfrak{D}^\theta)$.
\end{lemma}
\begin{corollary}\label{CL1} 
Let $M$ be a proper pointwise semi-slant submanifold of an $LCK$-manifold  $\tilde M$. Then  the distribution ${\mathfrak{D}}$ defines a totally geodesic foliation  if and only if
\label{2.14}
$$A_{FZ}J X-A_{FPZ}X =g(\beta^{\#},FZ)X-g(\alpha^{\#},FZ)JX,$$
for any $X\in\Gamma (\mathfrak{D})$ and $Z\in\Gamma (\mathfrak{D}^\theta)$.
\end{corollary}

For leaves of the pointwise slant distribution $\mathfrak{D}^\theta$, we have the following result.
 
\begin{lemma}\label{L3.3}
 Let $M$ be a pointwise semi-slant submanifold  of an $LCK$-manifold $\tilde M$ with proper pointwise slant distribution $\mathfrak{D}^\theta$. Then we have
\begin{align*}
g(\nabla_ZW,X)=\csc^2\theta g(A_{FPW}X-A_{FW}JX,Z)- g(Z,W)g(\alpha^{\#},X),
\end{align*} 
for any $X\in\Gamma (\mathfrak{D})$ and $Z, W\in\Gamma (\mathfrak{D}^\theta)$.
\end{lemma}
\begin{proof} For any for any $X\in\Gamma (\mathfrak{D})$ and $Z,W\in\Gamma (\mathfrak{D}^\theta)$, we find
\begin{align*}
g(\nabla_ZW, X)=g(J\tilde\nabla_Z W,JX)=g(\tilde\nabla_Z J W,J X)-g((\tilde\nabla_ZJ )W,JX).
\end{align*}
Using (\ref{2.2}), we get
\begin{align*}
g(\nabla_ZW, X)=&g(\tilde\nabla_ZPW,J X)+g(\tilde\nabla_ZFW,J X)-g(JZ,W)g(\lambda,JX)\\
&-g(Z,W)g(J\lambda,JX).
\end{align*}
Then, we derive
\begin{align*}
g(\nabla_ZW, X)=&-g(J\tilde\nabla_ZPW, X)-g(A_{FW}Z,J X)-g(PZ,W)g(\lambda,JX)\\
&-g(Z,W)g(\lambda,X).
\end{align*}
From the definition of covariant derivative of $J$ and the symmetry of the shape operator, we obtain
\begin{align*}
g(\nabla_ZW, X)=\,&g((\tilde\nabla_ZJ)PW, X)-g(\tilde\nabla_ZJ PW, X)-g(A_{FW}J X, Z)\notag\\
&-g(PZ,W)g(\lambda,JX)-g(Z,W)g(\lambda,X).
\end{align*}
Again using (\ref{2.2}) and (\ref{2.6}), we derive
\begin{align*}
g(\nabla_ZW, X)=\,&g(PZ, PW)g(\lambda,X)+g(Z,PW)g(J\lambda,X)
-g(\tilde\nabla_ZP^2W, X)\notag\\
&-g(\tilde\nabla_ZFPW, X)-g(A_{FW}J X, Z)-g(PZ,W)g(\lambda,JX)\\
&-g(Z,W)g(\lambda,X).\notag
\end{align*}
From the relation (\ref{2.8}), we find
\begin{align*}
g(\nabla_ZW, X)=\,&\cos^2\theta g(\tilde\nabla_ZW, X)-\sin2\theta X(\theta) g(W,X)+g(A_{FPW}Z, X)\notag\\
&+\cos^2\theta g(Z, W)g(\lambda,X)-g(A_{FW}J X, Z)-g(Z,W)g(\lambda,X).\notag
\end{align*}
By the orthogonality of two distributions, we derive
\begin{align*}
g(\nabla_ZW,X)=\csc^2{\theta}[g(A_{FPW}X,Z)-g(A_{FW}J X,Z)]-g(Z,W)g(\lambda,X),
\end{align*}
which proves the lemma completely. 
\end{proof}
Lemma \ref{L3.3} implies the following result.

\begin{corollary} \label{CL2}Let $M$ be a proper pointwise semi-slant submanifold of an $LCK$-manifold  $\tilde M$. Then the slant distribution ${\mathfrak{D}}^\theta$ defines a totally geodesic foliation if and only if
\begin{equation*}\label{2.16}
g(A_{FPZ}X-A_{FZ}J X, W)=\sin^2{\theta}g(\alpha^{\#},X)g(Z,W)
\end{equation*}
for any $X\in\Gamma (\mathfrak{D})$ and $Z,W\in\Gamma (\mathfrak{D}^\theta)$.
\end{corollary}

\begin{lemma}\label{L3.4} Let $M$ be a proper pointwise semi-slant submanifold of an $LCK$-manifold $\tilde M$. Then we have
$$\sin^2\theta\,g([Z, W], X)=g(A_{FZ}J X-A_{FPZ}X, W)-g(A_{FW}J X-A_{FPW}X, Z)$$
for any $X\in\Gamma (\mathfrak{D})$ and $Z,W\in\Gamma (\mathfrak{D}^\theta)$.
\end{lemma}
\begin{proof}  
From Lemma \ref{L3.3}, we have
\begin{equation}
\label{2.17}
\sin^2\theta\,g(\tilde\nabla_ZW, X)=g(A_{FPW}X-A_{FW}J X, Z)-\sin^2\theta\,g(\lambda,X)g(Z, W)
\end{equation}
for any $X\in\Gamma (\mathfrak{D})$ and $Z,W\in\Gamma (\mathfrak{D}^\theta)$. By interchanging $Z$ and $W$ in (\ref{2.17}), we find
\begin{equation}
\label{2.18}
\sin^2\theta\,g(\tilde\nabla_WZ, X)=g(A_{FPZ}X-A_{FZ}J X, W)-\sin^2\theta\,g(\lambda,X)g(Z, W).
\end{equation}
Thus, after subtracting (\ref{2.17}) from (\ref{2.18}), we get the required result.
\end{proof}

Now, we give the following integrability theorem.

\begin{theorem}
Let $M$ be a pointwise semi-slant submanifold of an $LCK$-manifold $(\tilde{M}, J, g)$. Then we have
\begin{enumerate}
\item [(i)]  The holomorphic distribution ${\mathfrak{D}}$ of $M$ is integrable if and only if
\begin{align*}
h(JY, X ) - h(JX, Y) =2 g(J X, Y )\alpha^{\#},\;\;\forall X,Y\in\Gamma (\mathfrak{D}).
\end{align*}
\item [(ii)]The pointwise slant distribution $\mathfrak{D}^{\theta}$ of $M$ is integrable if and only if
 \begin{align*}
g(A_{FW}JX-A_{FPW}X,Z)=g(A_{FZ}JX-A_{FPZ}X,W),\;\forall X\in\Gamma (\mathfrak{D}),\; Z,W\in\Gamma (\mathfrak{D}^{\theta}).
 \end{align*}
\end{enumerate}
\end{theorem}
\begin{proof} Let $M$ be a proper pointwise semi-slant submanifold of an $LCK$-manifold. Then,
  by interchanging $X$ and $Y$  and using symmetry of $A$,  we find from  Lemma \ref{L3.1} that
\begin{align*}
\sin^2\theta g([X,Y],Z)=g(A_{FZ}JY,X)-g(A_{FZ}JX,Y)-2g(JX,Y)g(\lambda,FZ)
\end{align*}
for any $X\in\Gamma({\mathfrak{D}})$ and $Z\in\Gamma({\mathfrak{D}}^\theta)$. 
Thus the distribution $\mathfrak{D}$ is integrable if and only if $g([X,Y],Z) = 0$ for all  $X\in\Gamma({\mathfrak{D}})$ and $Z\in\Gamma({\mathfrak{D}}^\theta)$, i.e., 
\begin{align*}
g(h(JY,X),FZ)-g(h(JX,Y),FZ)=2g(JX,Y)g(\lambda,FZ).
\end{align*}
Hence, (i) follows from the last relation. In a similar way, we can prove (ii).   
\end{proof}


\section{Pointwise semi-slant warped products: $N^{T}\times_fN^{\theta}$}\label{Sec4}

Let $N_1$ and $N_2$ be two Riemannian manifolds with Riemannian metrics $g_1$ and $g_2$, respectively, and $f$ be a positive differential function on $N_1$. Consider the product manifold $N_1\times N_2$ with its natural projections $\pi_1:N_1\times N_2\rightarrow N_1$ and $\pi_2:N_1\times N_2\rightarrow N_2$. Then the {\it warped product manifold} $N_1\times _fN_2$ is the product manifold $N_1\times N_2$ equipped with the warped product metric $g$ defined by
\begin{align*}
g(X, Y)=g_1({\pi_1}_\star X, {\pi_1}_\star Y)+(f\circ\pi_1)^2 g_2({\pi_2}_\star X, {\pi_2}_\star Y)
\end{align*}
for  $X, Y\in\Gamma(TM)$, where ${\pi_{i}}_{\star}$ is the tangent map of $\pi_{i}$. The function $f$ is called the {\it warping function} on $M$. A warped product manifold $N_1\times _fN_2$ is called {\it{trivial}} if its warping function $f$ is constant.

The following lemma is well-known.

\begin{lemma}\label{WL1}\cite{Bi} Let $M=N_1\times{_{f} N_2}$ be a warped product manifold with the warping function $f$, then for any $X, Y\in T(N_1)$ and $Z, W\in
T(N_2)$, we have
\begin{enumerate}
\item [{(i)}] $\nabla_XY\in T(N_1)$,
\item [{(ii)}] $\nabla_XZ=\nabla_ZX=X(\ln f)Z$,
\item [{(iii)}] $\nabla_ZW=\nabla_Z^{N_2}W-g(Z,W)\vec\nabla\ln f$,
\end{enumerate}
where $\nabla$ and $\nabla^{N_2}$ denote the Levi-Civita
connections on $M$ and $N_2$, respectively and $\vec\nabla\ln f$ is the
gradient of the function $\ln f$ defined as $g(\vec\nabla f, X)=X(f)$.
\end{lemma}

\begin{remark}
It is important to note that for a warped product $N_1\times_fN_2$, $N_1$ is totally geodesic and $N_2$ is totally umbilical in $M$ (c.f., \cite{Bi,C01a}).
\end{remark}

In this  section,  we study pointwise semi-slant warped products $M=N^T\times_fN^\theta$ in an $LCK$-manifold $\tilde M$ under the assumption that the Lee vector field $\alpha^{\#}$ is tangent to $M$. Clearly, $CR$-warped products and semi-slant warped product submanifolds are special cases of pointwise semi-slant warped product submanifolds $N^T\times_fN^{\theta}$ such that the slant function $\theta$ satisfies $\theta=\frac{\pi}{2}$ and $\theta={\rm constant}$, respectively.
For simplicity, we denote the tangent spaces of $N^T$ and $N^\theta$ by ${\mathfrak{D}}$ and ${\mathfrak{D}}^\theta$, respectively.

\begin{proposition}\label{WP1}
For a proper pointwise semi-slant warped product  $N^{T}\times {_{f}}N^{\theta }$ in  an $LCK$-manifold   $\tilde{M}$, the Lee vector field $\alpha^{\#}$ is orthogonal to ${\mathfrak{D}}^\theta$.
\end{proposition}
 
\begin{proof}
For any  $X\in\Gamma({\mathfrak{D}})$ and $Z\in\Gamma({\mathfrak{D}}^\theta)$, we have
\begin{align*}
g(h(X,Y), FZ)&=g(h(X, Y), JZ)=-g(J\tilde\nabla_XY, Z)-g(\nabla_XY, JZ)
\end{align*}
Using the definition of covariant derivative of $J$ and Lemma \ref{WL1}(i), we find
\begin{align}\label{S1}
g(h(X,Y), FZ)&=g((\tilde\nabla_XJ)Y),Z)-g(\tilde\nabla_XJY,Z)\notag\\
&=g(X,Y)g(J\lambda,Z)+g(JX,Y)g(\lambda,Z).
\end{align}
Since $h(X,Y)$ is symmetric with respect to $X$ and $Y$, we find $ g(JX,Y) g(\lambda, Z)=0$, which implies $g(\lambda, Z)=0$ for any $Z\in \Gamma ({\mathfrak{D}}^\theta)$. 
\end{proof}

\begin{remark}\label{R4}  Proposition \ref{WP1} shows that in our case the Lee vector field $ \alpha^{\#}$ is in $ \mathfrak{D}$.
\end{remark}

\begin{remark}\label{R5} For a proper CR-product of an $LCK$-manifold, the Lee vector field $ \alpha^{\#}$ is normal to  $ \mathfrak{D^{\perp}}$ (see \cite{BK04}).
\end{remark}

Now, we prove the following useful lemma.

\begin{lemma}\label{WL2}
Let $M=N^{T}\times {_{f}}N^{\theta }$ be a pointwise semi-slant warped product submanifold of an $LCK$-manifold $\tilde{M}$, where $N^{T}$ and $N^{\theta}$ are holomorphic and proper pointwise slant submanifolds of $\tilde{M}$, respectively and the Lee vector field $ \lambda$ is tangent to $M$. Then, we have

\vskip.05in
\noindent {\rm (i)} $g(h(X, Y), FZ)=0$,

\vskip.05in
\noindent{\rm (ii)} $g(h(X, Z), FW)=[g(\lambda,JX)-JX(\ln f)] g(Z, W)+[g(\lambda,X)-X(\ln f)]g(Z, PW)$
for any  $X, Y\in\Gamma({\mathfrak{D}})$ and $Z, W\in\Gamma({\mathfrak{D}}^\theta)$.
\end{lemma}
\begin{proof} From \eqref{S1} with the symmetry of $h$, we have
\begin{align}\label{S2}
g(h(X,Y), FZ)=g(X,Y)g(J\lambda,Z)+g(X,JY)g(\lambda, Z),
\end{align}
for any  $X, Y\in\Gamma({\mathfrak{D}})$ and $Z\in\Gamma({\mathfrak{D}}^\theta)$. Thus, it follows from \eqref{S1} and \eqref{S2} that
\begin{align*}
g(h(X, Y), FZ)=g(X, Y) g(\beta^{\#}, Z).
\end{align*}
Hence, the first part of the lemma follows from above relation by using Proposition \ref{WP1} and the fact that $\lambda$ is tangent to $M$. For the second part, we have 
\begin{align*}
g(h(X, Z), FW)=g(h(X, Z), JW)=-g(J(\tilde\nabla_ZX-\nabla_ZX), W),
\end{align*}
for any  $X\in\Gamma({\mathfrak{D}})$ and $Z,W\in\Gamma({\mathfrak{D}}^\theta)$. Thus, from the covariant derivative property of $J$, we find
\begin{align*}
g(h(X, Z), FW)=g((\tilde\nabla_ZJ)X,W)-g(\tilde\nabla_ZJX,W)+g(J\nabla_ZX, W).
\end{align*}
Using Lemma \ref{WL1}(ii), we get
\begin{align*}
g(h(X, Z), FW)=g((\tilde\nabla_ZJ)X,W)-JX(\ln f)g(Z,W)+X(\ln f)g(JZ,W).
\end{align*}
Now, using (\ref{2.2}) we find 
\begin{align*}
g(h(X, Z), FW)&=
[g(\lambda,JX)-JX(\ln f)]g(Z,W)+[X(\ln f)-g(\lambda,X)]g(PZ,W),\notag
\end{align*}
which  proves statement (ii).
\end{proof}
The following relations are obtained easily by interchanging $X$ with $JX$, $Z$ with $PZ$, and $W$ with $PW$ in Lemma \ref{WL2}(ii).
\begin{equation}\begin{aligned}\label{S3}
g(h(X, PZ), FW)=\,&\big[JX(\ln f)-g(\lambda,JX)\,\big]g(Z,PW)\\
& +\cos^2\theta[g(\lambda,X)-X(\ln f)\big]g(Z, W),
\end{aligned}\end{equation}
\begin{equation}\begin{aligned}\label{S4}
g(h(X, Z), FPW)=\,&\cos^2\theta\big[X(\ln f)-g(\lambda,X)\big]g(Z, W)\\
& +[g(\lambda,JX)-JX(\ln f)\,]g(Z, PW),
\end{aligned}\end{equation}
\begin{equation}\begin{aligned}\label{S5}
g(h(X, PZ), FPW)=\,&\cos^2\theta[g(\lambda,JX)-JX(\ln f)]g(Z,W)\\
&+\cos^2\theta[g(\lambda,X)-X(\ln f)] g(Z,PW),
\end{aligned}\end{equation}
\begin{equation}\begin{aligned}\label{3.3}
g(h(JX, Z), FW)=\,&[X(\ln f)-g(\lambda,X)]g(Z,W)\\
& +[g(\lambda,JX)-JX(\ln f)]g(Z,PW),
\end{aligned}\end{equation}
\begin{equation}\begin{aligned}\label{3.4}
g(h(JX, PZ), FW)=\,&[g(\lambda,X)-X(\ln f)\,]g(Z,PW)+\\
& \cos^2\theta[g(\lambda,JX)-JX(\ln f)]g(Z, W),
\end{aligned}\end{equation}
\begin{equation}\begin{aligned}\label{3.5}
g(h(JX, Z), FPW)=\,&\cos^2\theta[JX(\ln f)-g(\lambda,JX)]g(Z, W)\\
& +[X(\ln f)-g(\lambda,X)] g(Z, PW),
\end{aligned}\end{equation}
\begin{equation}\begin{aligned}\label{3.6}
g(h(JX, PZ), FPW)=\,&\cos^2\theta[X(\ln f)-g(\lambda,X)]g(Z, W)\\
& +\cos^2\theta[g(\lambda,JX)-JX(\ln f)\,]g(Z,PW).
\end{aligned}\end{equation}

Now, we give the following result for later use.

\begin{corollary}\label{WC1}
Let $M=N^{T}\times {_{f}}N^{\theta }$ be a nontrivial warped product pointwise semi-slant submanifold of an $LCK$-manifold $\tilde{M}$ and Lee field tangent to $M$. Then we have
\begin{align*}
g(h(X, PZ), FW)=-g(h(X, Z), FPW)
\end{align*}
for any $X,Y\in\Gamma({\mathfrak{D}})$ and $Z,W\in\Gamma({\mathfrak{D}}^\theta)$.
\end{corollary}
\begin{proof}
The proof follows from \eqref{S3} and \eqref{S4}.
\end{proof}

For a proper pointwise semi-slant warped product  $M=N^{T}\times {_{f}}N^{\theta }$ in an $LCK$-manifold $\tilde{M}$, let $\nu$ denote the invariant subbundle of  $T^{\perp}M$ which is the orthogonal complement of $F{\mathfrak D}^{\theta}$ in $T^{\perp}M$ so that 
\begin{align}\label{nu} T^{\perp}M= F{\mathfrak D}^{\theta}\oplus \nu.\end{align}

\begin{theorem}\label{WT1}
On a proper pointwise semi-slant warped product  $M=N^{T}\times {_{f}}N^{\theta }$ in an $LCK$-manifold $\tilde{M}$, if $ h(X,Z) \in\nu$ for any  $X\in\Gamma({\mathfrak{D}})$ and $Z\in\Gamma({\mathfrak{D}}^\theta)$, then we have $\alpha(X)=X(\ln f)$, where $\alpha$ is the Lee form.
\end{theorem}
 \begin{proof}
 By virtue of (\ref{WL2})(i) and the hypothesis of the theorem, we have
\begin{align}\label{S7}
[g(\lambda,JX)-JX(\ln f)]g(Z, W)+[g(\lambda,X)-   X(\ln f)]g(Z, PW)=0,
\end{align}
for any $X\in\Gamma({\mathfrak{D}})$ and $Z, W\in\Gamma({\mathfrak{D}}^\theta)$. Also, from \eqref{3.5} and the hypothesis of the theorem, we derive
\begin{align}\label{S8} 
\cos^2\theta[JX(\ln f)-g(\lambda,JX)]g(Z, W)+[X(\ln f)-g(\lambda,X)] g(Z, PW)=0.
\end{align}
Hence, it follows from \eqref{S7} and \eqref{S8} that
\begin{align}\label{S9} 
\sin^2\theta[g(\lambda,JX)-JX(\ln f)]g(Z, W)=0.
\end{align}
Since $M$ is proper pointwise semi-slant and $g$ is the Riemannian metric, the desired result follows from \eqref{S9}.
\end{proof}

\begin{corollary}\label{WC2}
Let $N^{T}\times {_{f}}N^{\theta }$ be a mixed totally geodesic pointwise semi-slant warped product in an $LCK$-manifold $\tilde{M}$. Then $\alpha(X)=X(\ln f)$ for any  $X\in\Gamma({\mathfrak{D}})$.
\end{corollary}

By using Lemma \ref{WL2}, we deduce the following result.

\begin{theorem}\label{WT2}
 Let $M=N^{T}\times {_{f}}N^{\theta }$ a proper pointwise semi-slant warped product in an $LCK$-manifold $\tilde{M}$ and the Lee vector field $ \alpha^{\#}$ is tangent to $M$. Then
 \begin{enumerate}
\item [{\rm (i)}] $g(A_{FZ} X, Y)=0$.
\item [{\rm (ii)}]  $g(A_{FZ}JX-A_{FPZ}X, W)=\sin^2\theta[X(\ln f)-\alpha(X)] g(Z, W)$
\end{enumerate}
for any  $X,Y\in\Gamma({\mathfrak{D}})$ and $Z, W\in\Gamma({\mathfrak{D}}^\theta)$.
\end{theorem}
\begin{proof} The first part is nothing but Lemma \ref{WL2} (i). The second part follows from \eqref{S4} and \eqref{3.3}.
\end{proof}

\begin{corollary}\label{WC3}
There do not exist a mixed totally geodesic CR-warped product submanifold of the form $N^{T}\times {_{f}}N^{\bot}$ in a Kaehler manifold $\tilde{M}$.
\end{corollary}
\begin{proof} Follows from Theorem \ref{WT2}(ii).
\end{proof} 

\section{Characterizations theorems}\label{Sec5}

Now, we provide a characterization of pointwise semi-slant warped products.

\begin{theorem}\label{TC1}
 Let $M=N^{T}\times {_{f}}N^{\theta }$ a proper pointwise semi-slant warped product in an $LCK$-manifold $\tilde{M}$ with its Lee vector field $ \alpha^{\#}$ tangent to $M$. Then we have
\begin{align}\label{SS}
X(\ln f)=\alpha(X)+\tan\theta\;X(\theta)
\end{align}
for any $X\in\Gamma({\mathfrak{D}})$.
\end{theorem}
\begin{proof}
For any  $X\in\Gamma({\mathfrak{D}})$ and $Z, W\in\Gamma({\mathfrak{D}}^\theta)$, we have
\begin{align*}
g(h(X,PZ), FW)=g(\tilde\nabla_XPZ, FW)=g(\tilde\nabla_XPZ, JW)-g(\tilde\nabla_XPZ, PW).
\end{align*}
From the covariant derivative property of $J$, we obtain
\begin{align*}g(A_{FW}PZ, X)=g((\tilde\nabla_XJ)PZ,W)-g(\tilde\nabla_XJPZ,W)-X(\ln f) \cos^{2}\theta g(Z,W).
\end{align*}
 Using (\ref{2.2}) and orthogonality of vector fields, we find
\begin{align*}g(A_{FW}PZ, X)&=-g(\tilde\nabla_XP^{2}Z,W)-g(\tilde\nabla_XFPZ,W)-X(\ln f) \cos^{2}\theta g(Z,W)\\
&= X(\ln f) \cos^{2}\theta g(Z,W)-\sin 2\theta X(\theta)g(Z,W)+g(A_{FPZ}X, W)\notag\\
&\;\;\;-X(\ln f)\cos^{2}\theta g(Z,W),
\end{align*}
which implies that
\begin{align}\label{SS1}
g(A_{FPZ}X, W)-g(A_{FW}PZ, X)=\sin2\theta\;X(\theta)\;g(Z, W).
\end{align}
On the other hand, from \eqref{S3} and \eqref{S4}, we find
\begin{align}\label{SS2}
g(A_{FPZ}X, W)-g(A_{FW}PZ, X)=2\cos^2\theta\;\left(X(\ln f)-\alpha(X)\right)\;g(Z, W).
\end{align}
Thus, equation \eqref{SS} follows from \eqref{SS1} and \eqref{SS2}.
\end{proof}

In order to prove another characterization for pointwise semi-slant  warped products, recall the following well-known Hiepko's Theorem.

\begin{theorem}\label{HT}
 \cite{Hi} Let ${\mathfrak{D}}_1$ and ${\mathfrak{D}}_2$ be two orthogonal distribution on a Riemannian manifold $M$. Suppose that both  ${\mathfrak{D}}_1$ and ${\mathfrak{D}}_2$ are involutive such that ${\mathfrak{D}}_1$ is a totally geodesic foliation and ${\mathfrak{D}}_2$ is a spherical foliation. Then $M$ is locally isometric to a non-trivial warped product $M_1\times_fM_2$, where $M_1$ and $M_2$ are integral manifolds of ${\mathfrak{D}}_1$ and ${\mathfrak{D}}_2$, respectively.
\end{theorem}

Now, we are able to prove another characterization of proper pointwise semi-slant
warped product submanifolds of the form $N^{T}\times_{f}N^{\theta}$.

\begin{theorem}\label{TC2}
Let $M$ be a proper pointwise semi-slant submanifold with invariant distribution $\mathfrak{D}$ and a proper pointwise slant distribution $\mathfrak{D}^{\theta}$ of an $LCK$-manifold $\tilde{M}$. Then $M$ is locally a warped product submanifold of the form $N^{T}\times_{f}N^{\theta}$ if and only if
\begin{align}\label{3.8}
A_{FZ}J X-A_{FPZ}X=\sin^2\theta\left(X(\mu)-\alpha(X)\right)Z,\,\;\forall\;X\in\Gamma({\mathfrak{D}}),\;Z\in\Gamma({\mathfrak{D}}^\theta),
\end{align}
for some smooth function $\mu$ on $M$ satisfying $W(\mu) = 0$ for any $W\in\Gamma({\mathfrak{D}}^\theta)$.
\end{theorem}
\begin{proof}
Let $M=N^{T}\times _{f}N^{\theta}$ be a pointwise semi-slant warped product submanifold of an $LCK$-manifold $\tilde{M}$. Then, by Theorem \ref{WT2}(i), we have $g(A_{FZ}{JX}, Y)=0$ for any $X, Y\in\Gamma(TN^T)$ and $Z\in\Gamma(TN^\theta)$, i.e., $A_{FZ}{\ J X}$ has no components in $TN^T$. Also, if we interchange $Z$ by $PZ$ in Theorem \ref{WT2}(i), then we get $g(A_{FPZ}{X}, Y)=0$, i.e., $A_{FPZ}{X}$ also has no components $TN^T$. Therefore, $A_{FZ}{\ J X}-A_{FPZ}{X}$ lies in $TN^\theta$ only. Applying this fact together with Theorem \ref{WT2}(ii), we obtain (\ref{3.8}) with $\mu=\ln f$ and $\alpha(X)=g(\lambda, X)$.

Conversely, suppose that $M$ is a proper pointwise  semi-slant submanifold of an $LCK$-manifold $\tilde{M}$ such that (\ref{3.8}) holds. Then it follows from Lemma \ref{L3.1} and the given condition (\ref{3.8}) that $\sin^2\theta\,g(\nabla_YX, Z)=0$ for $X, Y\in\Gamma({\mathfrak{D}})$ and $Z\in\Gamma({\mathfrak{D}}^\theta)$. Since $M$ is a proper pointwise semi-slant submanifold, $g(\nabla_YX, Z)=0$ holds. Therefore, the leaves of the distribution $\mathfrak{D}$ are totally geodesic in $M$.
 
On the other hand, it follows from condition (\ref{3.8}) and Lemma \ref{L3.4} that $\sin^2\theta\,g([Z, W], X)=0$ holds for any $X\in\Gamma({\mathfrak{D}})$ and $Z, W\in\Gamma({\mathfrak{D}}^\theta)$. Since $M$ is a proper pointwise semi-slant submanifold, we find $g([Z, W], X)=0$. Thus, the slant distribution ${\mathfrak{D}}^\theta$ is integrable. 

Now, let us consider the second fundamental form $h^\theta$ of a leaf $N^\theta$ of ${\mathfrak{D}}^\theta$ in $M$. Then, for any $Z, W\in\Gamma({\mathfrak{D}}^\theta)$ and $X\in\Gamma({\mathfrak{D}})$, we have
\begin{align*}
g(h^\theta(Z, W), X)=g(\nabla_ZW, X)=g(\tilde\nabla_ZW, X)=g(J\tilde\nabla_ZW, J X).
\end{align*}
Using the covariant derivative property of $J$ and the LCK-structure equation, we have
\begin{align*}
g(h^\theta(Z, W), X)&=g(\tilde\nabla_ZPW,J X)+g(\tilde\nabla_ZFW,J X)-g((\tilde\nabla_ZJ)W,J X)\\
&=g(\tilde\nabla_ZPW,J X)+g(\tilde\nabla_ZFW,J X)-g(JZ,W)g(\lambda,JX)\\
&\;\;\;-g(Z,W)g(\lambda,X).
\end{align*}
Again, by using the covariant derivative property of $J$, we find 
\begin{align*}
g(h^\theta(Z, W), X)&=-g(\tilde\nabla_ZJ PW, X)+g((\tilde\nabla_ZJ)PW, X)-g(A_{FW}Z, J X)\notag\\
&\;\;\;-g(PZ,W)g(\lambda,JX)-g(Z,W)g(\lambda,X).
\end{align*}
Therefore, by (\ref{2.6}) and (\ref{2.2}), we derive that
\begin{align*}
g(h^\theta(Z, W), X)&=-g(\tilde\nabla_ZP^2W, X)-g(\tilde\nabla_ZFPW, X)+g(PZ,PW)g(\lambda,X)\notag\\
&\;\;\;+g(Z,PW)g(J\lambda,X)-g(A_{FW}J X, Z)-g(PZ,W)g(\lambda,JX)\\
&\;\;\;-g(Z,W)g(\lambda,X).
\end{align*}
Then, from \eqref{2.8}  and \eqref{2.9}, we obtain
\begin{align*}
g(h^\theta(Z, W), X)&=\cos^2\theta\,g(\tilde\nabla_ZW, X)+\sin2\theta \,Z(\theta)g(W, X) +g(A_{FPW}Z, X)\notag\\
&\;\;\;\;+\cos^2\theta g(\lambda,X)g(Z, W)-g(A_{FW}J X, Z)-(\lambda,X)g(Z, W)\notag\\
&=\cos^2\theta\,g(\nabla_ZW, X)+g(A_{FPW}X-A_{FW}J X,Z)\notag\\
&\;\;\;\;-\sin^2{\theta}g(\lambda,X)g(Z,W).
\end{align*}
From the condition (\ref{3.8}), we find
$\sin^2\theta\,g(h^\theta(Z, W), X)=-\sin^2\theta\,X(\mu)g(Z,W).$ Hence,  we conclude from the definition of gradient that
$h^\theta(Z, W)=-\vec\nabla\mu g(Z,W)$,
which implies that $N^\theta$ is totally umbilical in $M$ with the mean curvature vector given by $H^\theta=-\vec\nabla\mu$. It is easy to see that the mean curvature vector $H^\theta$ is parallel to the normal connection $D^\#$ of $N^{\theta}$ in $M$. For this, let us  consider any $Y\in\Gamma(\mathfrak{D})$ and $Z\in\Gamma(\mathfrak{D}^{\theta})$, we find that 
\begin{align*}
g(D^\# _{Z}\vec\nabla\mu ,Y)&=g(\nabla _{Z}\vec\nabla\mu ,Y)
=Zg(\vec\nabla\mu ,Y)-g(\vec\nabla\mu ,\nabla _{Z}Y)\notag\\
&=Z(Y(\mu))-g(\vec\nabla\mu , [Z, Y])-g(\vec\nabla\mu ,\nabla _{Y}Z)\notag\\
&=Z(Y(\mu))-g(\vec\nabla\mu,ZY)+g(\vec\nabla\mu,YZ)+g(\nabla _{Y}\vec\nabla\mu ,Z)\notag\\
&=Z(Y(\mu))-Z(Y(\mu))+Y(Z(\mu))=0,\notag
\end{align*}
 since $Z(\mu )=0,\,\text{for all}\,Z\in \mathfrak{D}^{\theta}$ and thus $\nabla _{Y}\vec\nabla\mu\in \mathfrak{D}$. This means that the mean curvature of $N^{\theta }$ is parallel. Thus the leaves of $\mathfrak{D}^{\theta }$ are totally umbilical with parallel mean curvature. Consequently, the
spherical condition is satisfied.  Therefore, it follows from Theorem \ref{HT} that $M$ is a warped product submanifold $N^T\times_fN^\theta$ with warping function $\mu$. Hence the proof is complete.
\end{proof}

\section{Chen type inequality for pointwise semi-slant warped products}\label{Sec6}

In this section, we provide a sharp estimation for the length of the second fundamental form $h$ of an $m$-dimensional proper pointwise semi-slant warped product $M=N^T\times_fN^\theta$ in a $2k$-dimensional $LCK$-manifold $\tilde M$ such that the Lee vector field $\alpha^{\#}$ is tangent to $N^{T}$. Assume $\dim\; N^T=2p$ and $\dim\; N^\theta=2q$ so that we have $m=2p+2q$. 

Consider an orthonormal frame field $\{e_1,\cdots, e_p, e_{p+1}=J e_1,\cdots, e_{2p}=Je_p\}$ of  distribution $\mathfrak{D}$ and an orthonormal frame $\{e_{2p+1}=e_1^{\star},\cdots, e_{2p+q}=e_{q}^{\star}, e_{2p+1+q}=e_{q+1}^{\star}=\sec\theta Pe_{1}^{\star},\cdots, e_{m}=e_{2q}^{\star}=\sec\theta Pe_{q}^{\star}\}$  of the distribution $\mathfrak{D}^\theta$ on $M$.
 Then \begin{align*}&\{e_{m+1}=\tilde e_1=\csc\theta Fe_1^\star\cdots, e_{m+q}=\tilde e_q=\csc\theta Fe_q^\star, \\&\hskip.3in  e_{m+q+1}=\tilde e_{q+1}=\csc\theta\sec\theta FPe_1^\star,\cdots, e_{m+2q}=\tilde e_{2q}=\csc\theta\sec\theta FPe_q^\star\}\end{align*} is an orthonormal frame of  $F\mathfrak{D}^\theta$. Let $\{e_{m+2q+1}=\tilde e_{2q+1},\cdots, e_{2k}=\tilde e_{2k-m-2q}\}$ be an orthonormal frames of the invariant normal subbundle $\nu$ of $T^\perp M$ (cf. \eqref{nu}).
 
In the following, we will use the above frame fields to obtain a sharp estimation for the squared norm $\|h\|^2$ of the second fundamental form $h$ in terms of the gradient of the warping function $f$, Lee vector field $\alpha^{\#}$, and the slant function $\theta$ of the proper pointwise semi-slant warped product submanifold $M$ in the $LCK$-manifold $\tilde M$.
For simplicity, we denote  the tangential component of the Lee vector field $\alpha^{\#}$ of $\tilde M$ on the pointwise semi-slant warped product $M=N^T\times_fN^\theta$  by $\alpha^{\#}_{TM}$.

\begin{theorem}\label{IT1}
Let $M=N^T\times_fN^\theta$ be a pointwise semi-slant warped product submanifold of a locally conformal Kaehler manifold $\tilde M$ such that the Lee vector field  is tangent to $M$, where $N^T$ and $N^\theta$ are holomorphic and proper pointwise slant submanifolds of $\tilde M$, respectively. Then
\begin{enumerate}
\item [(i)] The squared norm of the second fundamental form $h$ of $M$ satisfies
 \begin{align} \label{ineq1}
\|h\|^2\geq 4q\left(\csc^2\theta+\cot^2\theta\right)\Big\{\|\vec\nabla(\ln f)\|^2+\Vert\alpha^{\#}_{TM} \Vert^{2}
-2G^{*}\Big\}, \end{align}
 where $G^{*}=\sum_{i=1}^{2p}g(\vec\nabla(\ln f), e_i)g(\alpha^{\#},e_i)$, $\Vert\alpha^{\#}_{TM} \Vert^{2}$ is the squared norm of $\alpha^{\#}_{TM}$, $q=\frac{1}{2}\dim\, N^T,\,p=\frac{1}{2}\dim\, N^\theta$,   and $\vec\nabla(\ln f)$ is the gradient of $\, \ln f$.
\item [(ii)] If the equality sign of \eqref{ineq1} holds identically, then $N^T$ is totally geodesic and $N^\theta$ is totally umbilical in $\tilde M$. Furthermore, $M$ is  minimal in $\tilde M$.
\end{enumerate}
\end{theorem}
\begin{proof} From the definition of $h$, we have
\begin{align*}
\|h\|^2=\sum_{i, j=1}^mg(h(e_i, e_j), h(e_i, e_j))=\sum_{r=m+1}^{2k}\sum_{i, j=1}^{m} g(h(e_i, e_j), e_r)^2.
\end{align*}
Decompose the above relation for $F\mathfrak{D}^\theta$- and $\nu$-components as follows
\begin{align} \label{ineq2}
\|h\|^2=\sum_{r=1}^{2q}\sum_{i, j=1}^{m} g(h(e_i, e_j), \tilde e_r)^2+\sum_{r=2q+1}^{2k-m-2q}\sum_{i, j=1}^{m}g(h(e_i, e_j), \tilde e_r)^2.
\end{align}
By computing $F\mathfrak{D}^\theta$-components terms and leaving $\nu$-components positive terms, we find
\begin{equation}\begin{aligned} \label{ineq3}
\|h\|^2&\geq\sum_{r=1}^{2q} \sum_{i, j=1}^{2p} g(h(e_i, e_j), \tilde
e_r)^2+2\sum_{r=1}^{2q}\sum_{i=1}^{2p}\sum_{ j=1}^{2q}
g(h(e_i, e_j^\star), \tilde e_r)^2\\
&+\sum_{r=1}^{2q} \sum_{i, j=1}^{2q} g(h(e_i^\star, e_j^\star),
\tilde e_r)^2.
\end{aligned}\end{equation}
In views of Lemma \ref{WL2}(i), the first term in the right hand side of \eqref{ineq3}
is identically zero and there is no relation for the last term $g(h(\mathfrak{D}^{\theta},\mathfrak{D}^{\theta}),F\mathfrak{D}^\theta)$. Thus compute just the second term of \eqref{ineq3}
\begin{align*}
\|h\|^2\geq 2\sum_{r=1}^{2q}\sum_{i=1}^{2p}\sum_{j=1}^{2q}g(h(e_i, e_j^\star), \tilde e_r)^2.
\end{align*}
It follows from the frame fields of $\mathfrak{D}$, $\mathfrak{D}^\theta$ and $F\mathfrak{D}^\theta$ chosen early in this section that
\begin{align*} &\hskip-.45in
\|h\|^2\geq  2\csc^2\theta\sum_{i=1}^{p}\sum_{r, j=1}^{q}g(h(e_i,
e_j^\star), Fe_r^\star)^2\notag\\
&+2\csc^2\theta\sec^4\theta\sum_{i=1}^{p}\sum_{r, j=1}^{q}g(h(e_i, Te_j^\star), FTe_r^\star)^2\notag
\\&+2\csc^2\theta\sum_{i=1}^{p}\sum_{r, j=1}^{q}g(h(J e_i, e_j^\star),
Fe_r^\star)^2\notag
\\&+2\csc^2\theta\sec^4\theta\sum_{i=1}^{p}\sum_{r, j=1}^{q}g(h( J e_i, Te_j^\star), FTe_r^\star)^2\notag
\\&+2\csc^2\theta\sec^2\theta\sum_{i=1}^{p}\sum_{r, j=1}^{q}g(h(e_i, Te_j^\star), Fe_r^\star)^2\notag\\
&+2\csc^2\theta\sec^2\theta\sum_{i=1}^{p}\sum_{r, j=1}^{q}g(h(e_i, e_j^\star), FTe_r^\star)^2\notag
\\&+2\csc^2\theta\sec^2\theta\sum_{i=1}^{p}\sum_{r, j=1}^{q}g(h(
J e_i, Te_j^\star), Fe_r^\star)^2\notag
\\&+2\csc^2\theta\sec^2\theta\sum_{i=1}^{p}\sum_{r, j=1}^{q}g(h(J e_i, e_j^\star), FTe_r^\star)^2.\notag
\end{align*}
Thus, we derive from Lemma \ref{WL2}(ii) and the relations \eqref{S3}-\eqref{3.6} that
\begin{align*}
\|h\|^2&\geq 4q\left(\csc^2\theta+\cot^2\theta\right)\sum\limits_{i=1}\limits^{p}\left(g(\lambda,Je_i )-Je_i(\ln f)\right)^2\notag\\
&\;\;\;+4q\left(\csc^2\theta+\cot^2\theta\right)\sum\limits_{i=1}\limits^{p}\left(g(\lambda, e_i )-e_i(\ln f)\right)^2\\
&= 4q\left(\csc^2\theta+\cot^2\theta\right)\Bigg\{\|\vec\nabla\ln f\|^2+\sum\limits_{i=1}\limits^{p}\left[ g(\alpha^{\#},e_i )^{2}+g(\alpha^{\#},Je_i )^{2}\right]\notag\\
&\;\;\;-2 \sum\limits_{i=1}\limits^{p}\left[(e_i\ln f)g(\alpha^{\#},e_i)+(Je_i\ln f)g(\alpha^{\#},Je_i)\right]\Bigg\}\\
&=4q\left(\csc^2\theta+\cot^2\theta\right)\Bigg\{\|\vec\nabla\ln f\|^2+\sum\limits_{i=1}\limits^{2p}\left[  g(\alpha^{\#},e_i )^{2}-2(e_i\ln f)g(\alpha^{\#},e_i)\right]\Bigg\}.
 \end{align*}
Since $\alpha^{\#} $ is in $\mathfrak{D}$ and orthogonal to $\mathfrak{D}^\theta$, the the above inequality takes the form  
\begin{align*}
\|h\|^2\geq 4q\left(\csc^2\theta+\cot^2\theta\right)\Bigg\{ \|\vec\nabla\ln f\|^2+\|\alpha^{\#}_{TM}\|^2 -2\sum\limits_{i=1}\limits^{2p}(e_i\ln f)g(\alpha^{\#},e_i)\Bigg\}\end{align*}
which is the inequality \eqref{ineq1}. 
If the equality sign of inequality \eqref{ineq1} holds, then it follows from the leaving $\nu$-components term in \eqref{ineq2} that
\begin{align}\label{eq1}
h(X, Y)\perp\nu,\;\;\;\forall\;X, Y\in\Gamma(TM).
\end{align}
Similarly, from the vanishing first term and the leaving third term in the right-hand-side of \eqref{ineq3}, we get 
\begin{align}\label{eq2}
h(\mathfrak{D},\mathfrak{D})\perp F\mathfrak{D}^\theta,\;\;\;h(\mathfrak{D}^\theta, \mathfrak{D}^\theta)\perp F\mathfrak{D}^\theta.
\end{align}
From \eqref{eq1} and  \eqref{eq2}, we conclude that 
\begin{align}\label{eq3}
h(\mathfrak{D},\mathfrak{D})=\{0\},\;\; h(\mathfrak{D}^\theta, \mathfrak{D}^\theta)=\{0\}\;\; {\mbox{and}}\;\; h(\mathfrak{D}, \mathfrak{D}^
\theta)\subset F\mathfrak{D}^\theta.
\end{align}
If $h^\theta$ denotes the second fundamental form of $N^\theta$ in $M$, then 
\begin{align}\label{eq4}
g(h^\theta(Z, W), X)=g(\nabla_ZW, X)=-(X(\ln f))\,g(Z, W)
\end{align}
for any  $X\in\Gamma({\mathfrak{D}})$ and $Z, W\in\Gamma({\mathfrak{D}}^\theta)$. 

Since $N^T$ is totally geodesic in $M$ \cite{Bi,book,book17}, using this fact together with \eqref{eq3}, we know that $N^T$ is totally geodesic in $\tilde M$. Also, since $N^\theta$ is totally umbilical in $M$ (cf. \cite{Bi,book}), it follows from \eqref{eq4} with this fact that $N^\theta$ is totally umbilical in $\tilde M$. Therefore, applying all conditions in \eqref{eq3} together with these facts, we conclude that $M$ is minimal in $\tilde M$. This completes the proof of the theorem.
\end{proof}

\section{Some Applications}\label{Sec7}

In this section, we provide some applications of our previous results.

If $\theta$ is constant in Theorem \ref{TC1}, then the warped product would be semi-slant warped product in a $LCK$-manifold. Such submanifolds have been studied by H. M. Tastan, S. G. Aydin in  \cite{HSP}. 

For semi-slant warped product submanifold of an $LCK$-manifold $\tilde M$,  Theorem \ref{TC1} gives the following.

\begin{theorem}
If $N^T\times_fN^\theta$ is a semi-slant warped product submanifold of an $LCK$-manifold $\tilde M$, then $X(\ln f)=\alpha(X)$ for any $X\in\Gamma(\mathfrak{D})$.
\end{theorem}

When $\tilde M$ is Kaehlerian, i.e. $\alpha^{\#}=0$, Theorem \ref{TC1} also implies the following.

\begin{theorem}\label{T:7.2}
There do not exist any proper semi-slant warped product submanifold of the form $M=N^T\times_fN^\theta$ in a Kaehler manifold $\tilde M$.
\end{theorem}

Theorem \ref{T:7.2} is the main result (Theorem 3.2) of \cite{Sahin06}. Therefore, Theorem \ref{TC1} also generalizes Theorem 3.2 of \cite{Sahin06}.

If $\theta={\pi}/{2}$ in Theorem \ref{TC2}, then $M$ is a CR-submanifold of an $LCK$-manifold. In this case, Theorem \ref{TC2} implies the following.

\begin{theorem}\label{T:7.3}
A CR-submanifold of an $LCK$-manifold $\tilde M$ with the Lee vector field $\lambda=\alpha^{\#}$ tangent to $M$ is a CR-warped product if and only if the Lee vector field is orthogonal to $\mathfrak{D}^\bot$ and the shape operator $A$ satisfies
\begin{align}\label{AVA10}
A_{JZ}X =-\left(g(J\lambda, X)+JX(\mu)\right)Z,
\end{align}
for each $X\in\Gamma(\mathfrak{D}),\;Z \in\Gamma(\mathfrak{D}^\bot)$, and for smooth function $\mu$ on $M$ with $W(\mu) = 0$, for each $W \in\Gamma(\mathfrak{D}^\bot)$.\end{theorem}

Theorem \ref{T:7.3} is the main result (Theorem 3.5) of \cite{Khan10}. Hence Theorem \ref{TC2} also generalizes the main result of \cite{Khan10}.

If we put $ \alpha^{\#}=0$ and $\theta =\dfrac{\pi}{2}$ in Theorem \ref{TC2},  the submanifold $M$ in Theorem \ref{TC2} is a CR-submanifold of a Kaehler manifold. Such submanifolds have been studied in \cite{C01a}. A characterization theorem for such manifolds is given by following.

\begin{theorem}\label{T:7.4}
A proper CR-submanifold $M$ of a Kaehler manifold $\tilde M$ is locally a CR-warped product if and only if its shape operator $A$ satisfies
\begin{align}\label{AC0}
A_{JZ}X = -JX(\mu)Z,\;X\in\Gamma(\mathfrak{D}),\;Z \in\Gamma(\mathfrak{D}^\bot),
 \end{align}
for some function $\mu$ on $M$ satisfying $W(\mu)=0,\;\forall\;W \in\Gamma(\mathfrak{D}^\bot)$, where $\mathfrak{D}$ and $\mathfrak{D}^\bot$ are the holomorphic and totally real distributions of $M$, respectively.
\end{theorem}
In fact, this theorem is Theorem 4.2 of \cite{C01a}. 
Further, if $\alpha^{\#}=0$ and $\theta$ is a slant function in Theorem \ref{TC2}, then  Theorem \ref{TC2} becomes the following characterization theorem (Theorem 5.1) of  \cite{Sahin13}.

\begin{theorem}\label{T:7.5}
Let $M$ be a pointwise semi-slant submanifold of a Kaehler manifold $\tilde M$. Then $M$ is locally a non-trivial warped product manifold of the
form $M = N^T\times_fN^\theta$ such that $N^\theta$ is a proper pointwise slant submanifold and $N^T$ is a holomorphic submanifold in $\tilde M$ if the following condition is
satisfied
\begin{align}\label{AS06}
 A_{FPW} X-A_{FW}JX = -(\sin^2\theta)X(\mu)W,\,\forall\;X\in\Gamma(\mathfrak{D}),\;W \in\Gamma(\mathfrak{D}^\theta),
 \end{align}
where $\mu$ is a function on $M$ such that $Z(\mu)=0$, for every $Z\in\Gamma(\mathfrak{D}^\theta)$.
\end{theorem}

In fact, in the relation (5.4) of Theorem 5.1 in \cite{Sahin13}, the term $(1 +\cos^2\theta)$ should be
$(1-\cos^2\theta)$, i.e., there is a missing term in that theorem.

If $\theta$ is constant on $M$ in Theorem \ref{TC2}, then $M$ is a semi-slant submanifold of an $LCK$-manifold. In this case, Theorem \ref{TC2} reduces to the following.

\begin{theorem}\label{T:7.6}
A  semi-slant submanifold $M$ of an $LCK$-manifold $\tilde M$ is locally a non-trivial warped product manifold of the
form $M = N^T\times_fN^\theta$ such that $N^\theta$ is a proper slant submanifold and $N^T$ is a holomorphic submanifold in $\tilde M$ if and only if the shape operator $A$ of $M$ satisfies
\begin{align}\label{AS00}
A_{FPZ}X-A_{FZ}JX=\sin^2\theta\left(g(\lambda, X)-X(\mu)\right)Z,\;\forall\; X\in\Gamma(\mathfrak{D}),\;Z\in\Gamma(\mathfrak{D}^\theta),
 \end{align}
for some smooth function $\mu$ on $M$ such that $W(\mu)=0$, for every $W \in\Gamma(\mathfrak{D}^\theta)$.
\end{theorem}

Now, we provide the following applications of Theorem \ref{IT1}.

If the slant function $\theta$ is a constant on $M$ in Theorem \ref{IT1}, then the pointwise semi-slant warped product submanifold $M$ reduces to a semi-slant warped product $M=N^T\times_fN^\theta$ in an $LCK$-manifold $\tilde M$ such that the Lee vector field is tangent to $M$. In this case, the inequality \eqref{ineq1} remains the same as
 \begin{align*}
\|h\|^2\geq 4q\left(\csc^2\theta+\cot^2\theta\right)\,\Big\{|\vec\nabla\ln f\|^2+ \Vert\alpha^{\#}_{TM} \Vert^{2}                                                       
-2G^{*}\Big\},
 \end{align*}
where $G^{*} =\sum_{i=1}^{2p}(e_i\ln f)g(\alpha^{\#},e_i).$

But, if we assume $\theta=\frac{\pi}{2}$ in Theorem \ref{IT1}, then the pointwise semi-slant warped product submanifold takes the form $M=N^T\times_fN^\bot$ i.e., $M$ is a CR-warped product, where $N^T$ and $N^\bot$ are holomorphic and totally real submanifolds of $\tilde M$, respectively. In this case, Theorem \ref{IT1} reduces to the following.

\begin{theorem}\label{AIT1}
Let $M=N^T\times_fN^\bot$ be a CR-warped product submanifold of a locally conformal Kaehler manifold $\tilde M$ such that the Lee vector field $\lambda=\alpha^{\#}$ is
tangent to $M$. Then
\begin{enumerate}
\item [(i)] The squared norm of the second fundamental form $h$ of $M$ satisfies
 \begin{align}
 \label{Aineq1}
\|h\|^2\geq 2q\Big\{\|\vec\nabla(\ln f)\|^2+ \Vert\alpha^{\#}_{TM} \Vert^{2}-2G^{*}\Big\}.
 \end{align}
\item [(ii)] If the equality sign of inequality \eqref{Aineq1} holds identically, then $N^T$ and $N^\bot$ are totally geodesic and totally umbilical submanifolds of $\tilde M$, respectively. Furthermore, $M$ is minimal in $\tilde M$.
\end{enumerate}
\end{theorem}

Theorem \eqref{AIT1} is the main result (Theorem 4.2) of \cite{BK04}. Hence, the main result of \cite{BK04} is a special case of Theorem \ref{IT1}.

A Kaehler manifold is an $LCK$-manifold with $\alpha^{\#}=0$.  Thus,  Theorem \ref{IT1} implies the following.
 
 \begin{theorem}\label{ATS13}
 Let $M=N^T\times_fN^\theta$ be non-trivial warped product pointwise semi-slant submanifold in a Kaehler manifold $\tilde M$. Then
\begin{enumerate}
\item [(i)] The squared norm of the second fundamental form of $M$ satisfies
 \begin{align} \label{Aineq2}
\|h\|^2\geq 4q\left(\csc^2\theta+\cot^2\theta\right)\,\|\vec\nabla^T(\ln f)\|^2
 \end{align}
where $q=\frac{1}{2} \dim\, N^\theta$.
\item [(ii)] If equality sign of inequality \eqref{Aineq2} holds, then $N^T$ and $N^\theta$ are totally geodesic and totally umbilical submanifolds of $\tilde M$, respectively. Furthermore, $M$ is not mixed totally geodesic in $\tilde M$.
\end{enumerate}
\end{theorem}

Theorem \ref{ATS13}  is a special case of Theorem \ref{IT1} which is Theorem 5.2.  of \cite{Sahin13}. 

If we choose $\alpha^{\#}=0$ and $\theta=\frac{\pi}{2}$, then Theorem \ref{IT1} reduces to Theorem 5.1 of \cite{C01a}, which is exactly the following well-known Chen's inequality for $CR$-warped product submanifolds in Kaehler manifolds.

\begin{theorem}\label{AC001}
Let $M=N^T\times_fN^\bot$ be a CR-warped product submanifold in a Kaehler manifold $\tilde M$. Then, the squared norm of the second fundamental form $h$ of $M$ satisfies
 \begin{align} \label{Aineq3}
 \|h\|^2\geq 2q\|\vec\nabla^{T}\ln f\|^2
\end{align}
where $q=\dim N^\bot$ and $\vec\nabla(\ln f)$ is gradient of $\ln f$. 
Furthermore, if the equality sign holds identically in the inequality, then $N^T$ and $N^\theta$ are totally geodesic and totally umbilical submanifolds of $\tilde M$, respectively. Moreover, $M$ is minimal in $\tilde M$.
\end{theorem}

\section{Examples}\label{Sec8}

Let $(x_1,\cdots,x_n,\, y_1,\cdots,y_n)$ be the Cartesian coordinates of a Euclidean $2n$-space  ${\mathbb{E}}^{2n}$ equipped with the Euclidean metric $g_{0}$. Then $\mathbb C^{n}=({\mathbb{E}}^{2n},J,g_{0})$ is a flat Kaehler manifold endowed with the canonical almost complex structure $J$ given by
\begin{align}\label{almost}
J(x_1,\cdots,x_n,\,y_1,\cdots,y_n)=(-y_1,\cdots,-y_n,\,x_1,\cdots,x_n).
\end{align}

The following result can be proved in the same way as Proposition 2.2 of \cite{C12}.

\begin{proposition}\label{P:8.1} Let $M=N^T\times_fN^\theta$ be a warped product pointwise semi-slant submanifold of a Kaehler manifold $\tilde M$. Then,   $M$ is also a warped product pointwise semi-slant submanifold  with the same slant function in a LCK-manifold $(\tilde M,J,\tilde g)$ with $\tilde g=e^{-f}g$, where $f$ is any smooth function on $\tilde M$.
\end{proposition}

\begin{example}\label{E8.1} \rm{Let $\mathbb C^{3}=(\mathbb E^{6},J,\,g_{0})$ be  a flat Kaehler manifold  defined above. Consider a 4-dimensional submanifold $M$ of ${\mathbb{C}}^{3}$ given by 
\begin{equation}\begin{aligned}\label{8.2}
&x_1=u,\; x_2=-ks\sin r,\; x_3=g(s),\; y_1=v,\; y_2=ks\cos r,\; y_3=h(s),
\end{aligned}\end{equation}
where $k$ is a positive number and $(g(s),h(s))$ is a unit speed planar curve.
Clearly, the tangent bundle $TM$ of $M$ is spanned by 
\begin{align*}
&X_1=\frac{\partial}{\partial u},\,\,\,X_2=\frac{\partial}{\partial v},\,\,\, X_3=\left(-ks\cos r \frac{\partial}{\partial x_2}-ks\sin r \frac{\partial}{\partial y_2}\right),\\
&X_4=\left(-k\sin s\frac{\partial}{\partial x_2}+g'(s)\frac{\partial}{\partial x_3}+k\cos  s\frac{\partial}{\partial y_2}+h'(s)\frac{\partial}{\partial y_3}\right).
\end{align*}
Further, $M$ is a proper semi-slant submanifold with the holomorphic distribution ${\mathfrak{D}}={\rm{Span}}\,\{X_1, X_2\}$ and the slant distribution ${\mathfrak{D}}^\theta={\rm Span}\{X_3,\,X_4\}$ with slant angle given by $\theta=\cos^{-1}({k}/{\sqrt{1+k^2}})$.
Obviously, both ${\mathfrak{D}}$ and ${\mathfrak{D}}^\theta$ are integrable and totally geodesic in $M$. Let $N^T$ and $N^\theta$ be integral manifolds of ${\mathfrak{D}}$ and ${\mathfrak{D}}^\theta$, respectively. Then the metric $\hat g$ on $M=N^{T}\times N^{\theta}$ induced from $\mathbb C^{3}$ is given by
\begin{align}\label{8.3} \hat g=g_{T}+g_{N^\theta},\;\;\; g_{T}=du^{2}+dv^{2},\;\; g_{N^\theta}=k^{2}s^{2} ds^{2} +(1+k^{2})ds^{2}.
\end{align}

Let $f=f(x_{1},y_{1})$ be a non-constant smooth function on $\mathbb C^{3}$ depending only on coordinates $x_{1}, y_{1}$ and let us consider the Riemannian metric $\tilde g = e^{-f}\,g_0$ on $\mathbb C^{3}$ conformal to  the standard metric $g_0$. Then $\tilde M=(\mathbb{E}^6, J,\, \tilde g)$ is a $GCK$-manifold and the metric on $M$ induced from the $GCK$-manifold  is the warped product metric:
\begin{align}\label{8.4} g_{M}= g_{N^{T}}+e^{-f}g_{N^{\theta}},\quad g_{N^{T}}=e^{-f} g_{T}.\end{align} 
Thus, it follows Proposition \ref{P:8.1} that  $(M,g_{M})$ is a proper warped product semi-slant submanifold in $\tilde M=(\mathbb{E}^6, J,\, \tilde g)$.
Now, since $f=f(x_{1},y_{1})$ is a smooth function on $\mathbb C^{3}$ depending only on coordinates $x_{1}, y_{1}$, it follows from \eqref{8.2} that, restricted to the submanifold $M$, the Lee form of $\tilde M$ is given by
\begin{align}\label{8.5} \alpha=df=\frac{\partial f}{\partial u_{1}}du_{1}+\frac{\partial f}{\partial u_{2}}du_{2}.\end{align}
Consequently, it follows from \eqref{8.4} and \eqref{8.5} that the Lee vector field $\alpha^{\#}$ is tangent to $N^T$, and hence it is tangent to $M$.}
\end{example}

\begin{example}\label{E8.2} \rm{Consider a 4-dimensional submanifold $M$ of ${\mathbb{C}}^{3}$ given by 
\begin{equation}\begin{aligned}\label{8.6}
&x_1=\frac{1}{2}(u_1^2+u_2^2),\;\; x_2=u_3\cos u_4,\;\;x_3=u_3\sin u_4,\\
&y_1=\frac{1}{2}(u_1^2-u_2^2),\;\;\, y_2=u_4\cos u_3,\;\;\, y_3=u_4\sin u_3,
\end{aligned}\end{equation}
 defined on an open subset of $\mathbb E^{4}$ with $u_{1},u_{2}\ne 0$, $u_{3}u_{4}\ne 1$ and $u_{3}- u_{4}\in (0,\frac{\pi}{2})$.
Then the tangent bundle $TM$ of $M$ is spanned by 
\begin{align*}
&X_1=\frac{1}{\sqrt2}\left(\frac{\partial}{\partial x_1}+\frac{\partial}{\partial y_1}\right),\,\,\,X_2=\frac{1}{\sqrt2}\left(\frac{\partial}{\partial x_1}-\frac{\partial}{\partial y_1}\right),\\
&X_3=\frac{1}{\sqrt{1+u_4^2}}\left(\cos u_4\frac{\partial}{\partial x_2}+\sin u_4\frac{\partial}{\partial x_3}-u_4\sin u_3\frac{\partial}{\partial y_2}+u_4\cos u_3\frac{\partial}{\partial y_3}\right),\\
&X_4=\frac{1}{\sqrt{1+u_3^2}}\left(-u_3\sin u_4\frac{\partial}{\partial x_2}+u_3\cos u_4\frac{\partial}{\partial x_3}+\cos u_3\frac{\partial}{\partial y_2}+\sin u_3\frac{\partial}{\partial y_3}\right).
\end{align*}
Then $M$ is a proper pointwise semi-slant submanifold such that the holomorphic distribution is given by ${\mathfrak{D}}={\rm{Span}}\,\{X_1, X_2\}$ and the proper pointwise slant distribution is ${\mathfrak{D}}^\theta={\rm{Span}}\{X_3,\,X_4\}$. It is direct to show that the slant function $\theta$ of ${\mathfrak{D}}^\theta$ satisfies $$\cos^{2}\theta=\frac{(u_3u_4-1)^{2}\cos^{2}(u_3-u_4)}{(1+u_3^2)(1+u_4^2)}.$$
Clearly, both ${\mathfrak{D}}$ and ${\mathfrak{D}}^\theta$ are integrable and totally geodesic in $M$. Let $N^T$ and $N^\theta$ be integral manifolds of ${\mathfrak{D}}$ and ${\mathfrak{D}}^\theta$, respectively. It is to see that the metric $\hat g$ on $M=N^{T}\times N^{\theta}$ induced from $\mathbb C^{3}$ is given by
\begin{align}\label{8.7} \hat g=g_{T}+g_{N^{\theta}},\end{align}
where 
\begin{align}\label{8.8}g_{T}=2u_{1}^{2}du_{1}^{2}+2u_{2}^{2}du_{2}^{2},\quad  g_{N^{\theta}}=(1+u_{4}^{2})du_{3}^{2}+(1+u_{3}^{2})du_{4}^{2}.\end{align}

Let $f=f(x_{1},y_{1})$ be a non-constant smooth function on $\mathbb C^{3}$ depending only on coordinates $x_{1}, y_{1}$ and  consider the Riemannian metric $\tilde g = e^{-f}\,g_0$ on $\mathbb C^{3}$ as in Example 8.1.  
Then the induced metric on $M$ from $\tilde M=(\mathbb{E}^6, J,\, \tilde g)$ is the warped product metric:
\begin{align}\label{8.9} g_{M}= g_{N^{T}}+e^{-f}g_{N^{\theta}},\quad g_{N^{T}}=e^{-f} g_{T}.\end{align} 
Now, it follows Proposition \ref{P:8.1} that  $(M,g_{M})$ is a proper warped product pointwise semi-slant submanifold in $\tilde M=(\mathbb{E}^6, J,\, \tilde g)$.
Now, since $f=f(x_{1},y_{1})$ is a smooth function on $\mathbb C^{3}$ depending only on coordinates $x_{1}, y_{1}$, it follows from \eqref{8.6} that, restricted to the submanifold $M$, the Lee form of $\tilde M$ is given by \eqref{8.5}.
Consequently, it follows from \eqref{8.5},  \eqref{8.8}, and \eqref{8.9} that the Lee vector field $\alpha^{\#}$ is tangent to $M$.}
\end{example}

\begin{remark} By applying the same method as Example \ref{E8.2}, we may construct many other examples of proper warped product pointwise semi-slant submanifolds $M$ in $LCK$-manifolds whose Lee vector fields are tangent to $M$.
\end{remark}


\end{document}